\documentclass[12pt]{article}
\usepackage{amssymb,amsmath,amsthm}
\usepackage{graphicx}
\newtheorem{theorem}{Theorem}
\newtheorem{proposition}[theorem]{Proposition}
\newtheorem{corollary}[theorem]{Corollary}
\newtheorem{conjecture}[theorem]{Conjecture}
\newtheorem{lemma}[theorem]{Lemma}

\begin{document}
\title{Irredundant Families of Subcubes}
\author{David Ellis}
\date{January 2010}
\maketitle

\begin{abstract}
We consider the problem of finding the maximum possible size of a family of \(k\)-dimensional subcubes of the \(n\)-cube \(\{0,1\}^{n}\), none of which is contained in the union of the others. (We call such a family `{\em irredundant}'). Aharoni and Holzman \cite{aharoni} conjectured that for \(k > n/2\), the answer is \({n \choose k}\) (which is attained by the family of all \(k\)-subcubes containing a fixed point). We give a new proof of a general upper bound of Meshulam \cite{meshulam}, and we prove that for \(k \geq n/2\), any irredundant family in which all the subcubes go through either \((0,0,\ldots,0)\) or \((1,1,\ldots,1)\) has size at most \({n \choose k}\). We then give a general lower bound, showing that Meshulam's upper bound is always tight up to a factor of at most \(e\).
\end{abstract}

\section{Introduction}
Let \(\{0,1\}^{n}\) denote the \(n\)-dimensional discrete cube, the set of all 0-1 vectors of length \(n\). A \(k\)-{\em dimensional subcube} (or \(k\)-{\em subcube}) of \(\{0,1\}^{n}\) is a subset of \(\{0,1\}^{n}\) of the form
\[\{x \in \{0,1\}^{n}:\ x_{i} = a_{i}\ \forall i \in T\}\]
where \(T\) is a set of \(n-k\) coordinates, called the \emph{fixed coordinates}, and the \(a_{i}\)'s are fixed elements of \(\{0,1\}\). The other coordinates \(S = [n]\setminus T\) are called the \emph{moving coordinates}. We will represent a subcube by an \(n\)-tuple of 0's, 1's and \(*\)'s, where the \(*\)'s denote moving coordinates and the 0's and 1's denote fixed coordinates. For example, \((*,*,*,0,1)\) denotes a 3-dimensional subcube of \(\{0,1\}^{5}\).

We consider the problem of finding the maximum possible size of a family of \(k\)-subcubes of the \(n\)-cube \(\{0,1\}^{n}\), none of which is contained in the union of the others. In other words, each has a vertex not contained in any of the others (which we call a `private' vertex). We will call such a family `{\em irredundant}', and we write \(M(n,k)\) for the maximum size of an irredundant family of \(k\)-subcubes of \(\{0,1\}^{n}\).

Let \([n]\) denote the set \(\{1,2,\ldots,n\}\). We may identify \(\{0,1\}^{n}\) with \(\mathbb{P}[n]\), the set all subsets of \([n]\), by identifying a subset \(x \subset [n]\) with its characteristic vector \(\chi_{x}\), defined by
\[\chi_{x}(i) = 1\ \forall i \in x,\ \chi_{x}(i) = 0\ \forall i \notin x.\]
We write \((0,0,\ldots,0)=\boldsymbol{0}\) and \((1,1,\ldots,1) = \boldsymbol{1}\). We will refer to \(|x \Delta y|\), the number of coordinates in which \(x\) and \(y\) differ, as the \emph{Hamming distance} between \(x\) and \(y\), and the set
\[\{y \in \{0,1\}^{n}: |x \Delta y| \leq r\}\]
as the \emph{Hamming ball of centre} \(x\) \emph{and radius} \(r\).

Here are some natural examples of irredundant families:\\
\\
The family of all translates of a fixed \(k\)-subcube,
\[\{A+x: x \in \{0,1\}^{n}\}\]
where \(A\) is a \(k\)-subcube of \(\{0,1\}^{n}\) --- in other words, the collection of all the subcubes having the same moving coordinates as \(A\). This family partitions \(\{0,1\}^{n}\), so every vertex is a private vertex of its subcube, and it is a maximal irredundant family; it has size \(2^{n-k}\).\\
\\
The family \(\mathcal{F}_{\boldsymbol{0}}\) of all \(k\)-subcubes containing \(\boldsymbol{0}\), \(\{\mathbb{P}x : x \in [n]^{(k)}\}\). Clearly, \(x\) is a private vertex of the \(k\)-subcube \(\mathbb{P}x\); it is the unique such, since any \(y \subsetneq x\) can be extended to a different \(k\)-set \(z \neq x\). This family has size \(n \choose k\). For \(k \geq \tfrac{1}{2}n\) it is maximal, since then any \(k\)-subcube contains a \(k\)-set. Similarly, for any \(v \in Q_{n}\) we let \(\mathcal{F}_{v}\) be the collection of all \(k\)-subcubes through \(v\); we call these the `principal' irredundant families. Aharoni and Holzman \cite{aharoni} conjectured that for \(k > n/2\), there are no larger irredundant families: 
\begin{conjecture}[Aharoni-Holzman, 1991]
\label{conjecture:aharoni}
If \(k > n/2\), any irredundant family of \(k\)-subcubes of \(\{0,1\}^{n}\) has size at most \({n \choose k}\).
\end{conjecture}

Aharoni and Holzman (unpublished -- see \cite{meshulam}) gave the following general upper bound on the maximum size of an irredundant family of \(k\)-subcubes of \(\{0,1\}^{n}\):
\begin{equation}
\label{eq:linearalgebrabound}
M(n,k) \leq \sum_{i=k}^{n}{n \choose i}\quad \forall k \leq n.
\end{equation}

This may be proved using a short linear independence argument. Mesulam \cite{meshulam} proved the following stronger upper bound using a purely combinatorial argument: 
\begin{equation}
\label{eq:bound1}
M(n,k) \leq \frac{2^n}{\sum_{i=0}^{k}{n \choose i}}{n \choose k} \quad \forall k \leq n.
\end{equation}

(Intuitively, this is saying that, if there were a partition of \(\{0,1\}^{n}\) into Hamming balls of radius \(k\), it would be best to take the irredundant family of all \(k\)-subcubes containing one of the centres of the balls.) We will give a simple proof of Meshulam's bound using Bollob\'as' Inequality. A variant of this proof shows that if we choose one private vertex for each subcube in an irredundant family, then any Hamming ball of radius \(k\) contains at most \({n \choose k}\) of these private vertices. (This immediately implies Meshulam's bound by averaging over all Hamming balls of radius \(k\).)

For \(k/n > \gamma\), where \(\gamma \in (\tfrac{1}{2},1)\) is fixed, Meshulam's bound gives \(M(n,k) \leq (1+o(1)) {n \choose k}\), i.e. it asymptotically approaches the conjectured bound; if \(\gamma \geq \gamma_{0} \approx 0.8900\), it gives \(M(n,k) < {n \choose k}+1\) for \(n\) sufficiently large, proving Conjecture \ref{conjecture:aharoni} in this case.

We observe that equality holds in Meshulam's bound when there is a partition of \(\{0,1\}^{n}\) into Hamming balls of radius \(k\), i.e. in the following cases:
\begin{itemize}
\item \(k=1\), \(n+1\) is a power of 2
\item \(k = 3\), \(n=23\)
\item \(n = 2k+1\)
\end{itemize}

When \(n=2k+1\), the irredundant family of all \(k\)-subcubes containing either \(\boldsymbol{0}\) or \(\boldsymbol{1}\) has size \(2{n \choose k}\).

We are then led to investigate the special case when every subcube must go through either \(\boldsymbol{0}\) or \(\boldsymbol{1}\); we prove by an unusual linear algebra argument that for \(k \geq n/2\), any irredundant family in which all \(k\)-subcubes go through either \(\boldsymbol{0}\) or \(\boldsymbol{1}\) has size at most \({n \choose k}\).

Finally, we obtain a general lower bound for all \(n\) and \(k\). A probabilistic argument shows that there exists an irredundant family of \(k\)-subcubes of \(\{0,1\}^{n}\) of size at least
\begin{equation}
\label{eq:randomlowerbound}
\beta(1-\beta)^{(1-\beta)/\beta}2^{n},
\end{equation}
where 
\[\beta := \frac{{n \choose k}}{\sum_{i=0}^{k}{n \choose i}}.\]
Combining this with Meshulam's bound, we see that
\[\beta(1-\beta)^{(1-\beta)/\beta} 2^{n} \leq M(n,k) \leq \beta 2^{n}.\]
The ratio between the upper and lower bound above is at most \(e\) for all \(n\) and \(k\).

If \(k = \lfloor \gamma n \rfloor\) for fixed \(\gamma \in (0,\tfrac{1}{2})\), then
\[\beta = \left(\frac{1-2\gamma}{1-\gamma}\right)(1+o(1)),\]
so we obtain
\[(1+o(1))\left(\frac{\gamma}{1-\gamma}\right)^{\frac{\gamma}{1-2\gamma}}\left(\frac{1-2\gamma}{1-\gamma}\right)2^{n} \leq M(n,\lfloor \gamma n \rfloor) \leq (1+o(1))\left(\frac{1-2\gamma}{1-\gamma}\right)2^{n},\]
showing that \(M(n,\lfloor \gamma n \rfloor)\) has order of magnitude \(2^{n}\).

If \(k = o(n)\), we obtain \(M(n,k) = (1-o(1))2^{n}\).

\section{Upper bounds}

Aharoni and Holzman proved the following:

\begin{proposition}[Aharoni-Holzman, 1991]
\label{proposition:aharoni}
For any \(k \leq n\), any irredundant family of \(k\)-subcubes of \(\{0,1\}^{n}\) has size at most
\[\sum_{i=k}^{n}{n \choose i}\]
\end{proposition}
\begin{proof}
Let \(C\) be a \(k\)-subcube of \(\{0,1\}^{n}\); we write \(0(C)\) for its set of fixed 0's and \(1(C)\) for its set of fixed 1's. The characteristic function \(\chi_{C}\) of \(C\) can be written as a function of \((x_{1},\ldots,x_{n}) \in \mathbb{R}^{n}\) as follows:
\begin{equation}\label{eq:chi}
\chi_{C}(x_{1},\ldots,x_{n}) = \prod_{i \in 0(C)}(1-x_{i})\prod_{i \in 1(C)}x_{i}
\end{equation}
---for example,
\[\chi_{(1,*,*,*,0)}(x_{1},x_{2},x_{3},x_{4},x_{5}) = x_{1}(1-x_{5}).\]

Now let \(\mathcal{A}\) be an irredundant family of \(k\)-subcubes of \(\{0,1\}^{n}\). Then
\[\{\chi_{C} : C \in \mathcal{A}\}\]
is a linearly independent subset of the vector space \(\mathbb{R}[x_{1},\ldots,x_{n}]\). To see this, for each \(C \in \mathcal{A}\), choose a private vertex \(w_{C} \in C\). Suppose
\[\sum_{C \in \mathcal{A}} a_{C}\chi_{C} = 0\]
for some real numbers \(\{a_{C}:\ C \in \mathcal{A}\}\). Then for any \(D \in \mathcal{A}\), evaluating the above on \(w_{D}\) gives:
\[0 = \sum_{C \in \mathcal{A}} a_{C}\chi_{C}(w_{D}) = a_{D}.\]
It is easy to check that the set of monomials
\[S = \{\prod_{i \in A}x_{i}: A \in [n]^{(\leq n-k)}\}\]
is a basis for the vector subspace
\[W = \langle \chi_{C}: C \textrm{ is a } k \textrm{-subcube of } \{0,1\}^{n} \rangle \subset \mathbb{R}[x_{1},\ldots,x_{n}].\]
Hence
\[|\mathcal{A}| \leq \dim(W) = |S| = \sum_{l=0}^{n-k}{n \choose l} = \sum_{i=k}^{n}{n \choose i},\]
proving the proposition.
\end{proof}

For \(k = \lfloor \gamma n \rfloor\), where \(\gamma \in (\tfrac{1}{2},1)\), we have:
\[\sum_{i=k}^{n}{n \choose i} = \sum_{l=0}^{n-k} {n \choose l} \leq \frac{3\gamma-1}{2\gamma-1}{n \choose \lfloor \gamma n \rfloor},\]
so Proposition \ref{proposition:aharoni} gives the correct order of magnitude.

For \(n=2k-1\), however, it only gives \(M(2k-1,k) \leq 2^{2k-2}\), compared with \(2(1-o(1)){2k-1 \choose k}\) from Meshulam's bound.

We now give a proof of Meshulam's bound which we believe to be slightly more intuitive than the proof in \cite{meshulam}. The idea is that for any irredundant family \(\mathcal{A}\) and any choice of private vertices, for every \(x \in \{0,1\}^{n}\), the private vertices chosen for the subcubes containing \(x\) cannot be too closely packed around \(x\). Our main tool is Bollob\'as' Inequality:

\begin{theorem}[Bollob\'as, 1965]
\label{thm:boll}
Let \(a_{1},\ldots,a_{N}\) and \(b_{1},\ldots,b_{N}\) be subsets of \(\{1,2,\ldots,n\}\) such that \(a_{i} \cap b_{j} = \emptyset\) if and only if \(i=j\). Then
\[\sum_{i=1}^{N} {|a_{i}|+|b_{i}| \choose |b_{i}|}^{-1} \leq 1\].
Equality holds only if there exists a subset \(Y \subset [n]\) and an integer \(a \in \mathbb{N}\) such that \(\{a_{1},\ldots,a_{N}\} = Y^{(a)}\), and \(b_{i} = Y \setminus a_{i}\ \forall i\).
\end{theorem}
For a proof, we refer the reader to \cite{bollobas}.

Given an irredundant family \(\mathcal{A}\), we will fix a choice of private vertices, and deduce from Theorem \ref{thm:boll} an inequality involving the subcubes containing a fixed vertex \(x \in Q_{n}\); we will then sum this inequality over all \(x \in Q_{n}\) to prove bound (\ref{eq:bound1}).

\begin{theorem}[Meshulam, 1992]
\label{thm:meshulambound}
For any \(k \leq n\), if \(\mathcal{A}\) is an irredundant family of \(k\)-subcubes of \(\{0,1\}^{n}\), then
\[|\mathcal{A}| \leq \frac{2^n}{\sum_{i=0}^{k}{n \choose i}}{n \choose k}\]
\end{theorem}
\begin{proof}
Let \(\mathcal{A}\) be an irredundant family of \(k\)-subcubes of \(\{0,1\}^{n}\), and for each subcube \(C \in \mathcal{A}\), choose a private vertex \(w_{C} \in C\).\\
\\
\textit{Claim:} For any \(x \in \{0,1\}^{n}\), 
\begin{equation}
\label{eq:boundsum1}
\sum_{C \in \mathcal{A}: x \in C}{|w_{C}\Delta x| + n-k \choose n-k}^{-1} \leq 1.
\end{equation}\\
\textit{Proof of Claim:}\\
This is an immediate consequence of Bollob\'as' Inequality. By symmetry, we may assume that \(x = \boldsymbol{0}\). Let \(\{C_{1},\ldots,C_{N}\}\) be the collection of subcubes in \(\mathcal{A}\) containing \(\boldsymbol{0}\). Each \(C_{i}\) is of the form \(\mathbb{P}v_{i}\) for some \(k\)-set \(v_{i}\). Let \(w_{i} = w_{C_{i}}\) be the private vertex chosen for \(C_{i}\). Notice that \(w_{i} \subset v_{j}\) if and only if \(i=j\), i.e. \(w_{i} \cap v_{j}^{c} = \emptyset\) if and only if \(i=j\), so applying Bollob\'as' Inequality gives:
\[\sum_{i=1}^{N}{|w_{i}| + |v_{i}^{c}| \choose |v_{i}^{c}|}^{-1} \leq 1,\]
i.e.
\begin{equation}
\label{eq:prob}
\sum_{i=1}^{N}{|w_{i}| + n-k \choose n-k}^{-1} \leq 1,
\end{equation}
proving the claim.

The inequality (\ref{eq:boundsum1}) expresses the fact that the private vertices chosen for the subcubes containing \(x\) cannot be too densely packed around \(x\). Summing (\ref{eq:boundsum1}) over all \(x \in \{0,1\}^{n}\), and interchanging the order of summation, we obtain:
\begin{eqnarray*}
2^{n} & \geq & \sum_{x \in \{0,1\}^{n}} \sum_{\substack{C \in \mathcal{A}:\\ x \in C}} {|w_{C}\Delta x| + n-k \choose n-k}^{-1}\\
&= & \sum_{C \in \mathcal{A}} \sum_{x \in C} {|w_{C}\Delta x| + n-k \choose n-k}^{-1}\\
& = & |\mathcal{A}|\sum_{l=0}^{k} \frac{{k \choose l}}{{l+n-k \choose n-k}}\\
& = & |\mathcal{A}|\sum_{l=0}^{k} \frac{k!(n-k)!l!}{l!(k-l)!(l+n-k)!}\\
& = & |\mathcal{A}|\frac{k!(n-k)!}{n!}\sum_{l=0}^{k} \frac{n!}{(k-l)!(n-(k-l))!}\\
& = & \frac{|\mathcal{A}|}{{n \choose k}}\sum_{l=0}^{k}{n \choose k-l}\\
& = & \frac{|\mathcal{A}|}{{n \choose k}}\sum_{l=0}^{k}{n \choose l}
\end{eqnarray*}
Hence,
\[|\mathcal{A}| \leq \frac{2^{n}}{\sum_{l=0}^{k}{n\choose l}}{n \choose k}\]
as required.
\end{proof}

As observed by Meshulam, for \(k \geq \tfrac{9}{10}n\), by standard estimates, the bound above is \(< {n \choose k} +1\), implying Conjecture \ref{conjecture:aharoni} in this case. More precisely, let
\[H_{2}(\gamma) = \gamma \log_{2}(1/\gamma)+(1-\gamma)\log_{2}(1/(1-\gamma))\]
denote the binary entropy function, and let \(\gamma_{0}\) be the unique solution of \(H_{2}(\gamma_{0})=\tfrac{1}{2}\) in \((\tfrac{1}{2},1)\), so that \(\gamma_{0} = 0.8900\) (to 4 d.p.); then we have the following 

\begin{corollary}
For \(n\) sufficiently large, and \(k \geq \gamma_{0}n\), any irredundant family of \(k\)-subcubes of \(\{0,1\}^{n}\) has size at most \({n \choose k}\).
\end{corollary}

In fact, Meshulam proved a generalization of Theorem \ref{thm:meshulambound} for irredundant families of \(k\)-dimensional subgrids of the \(n\)-dimensional grid \(\mathbb{Z}_{m}^{n}\). (A \(k\)-\textit{subgrid} of \(\mathbb{Z}_{m}^{n}\) is a subset of \(\mathbb{Z}_{m}^{n}\) the form
\[\{x \in \mathbb{Z}_{m}^{n}:\ x_{i}=a_{i}\ \forall i \in T\},\]
where \(T\) is a set of \(n-k\) coordinates, and the \(a_{i}\)'s are fixed elements of \(\mathbb{Z}_{m}\). A family of \(k\)-subgrids of \(\mathbb{Z}_{m}^{n}\) is said to be \textit{irredundant} if none of its subgrids is contained in the union of the others.) Meshulam proved the following:

\begin{theorem}[Meshulam, 1992]
Let \(\mathcal{A}\) be an irredundant family of \(k\)-subgrids of \(\mathbb{Z}_{m}^{n}\); then
\[|\mathcal{A}| \leq \frac{m^{n}}{\sum_{j=n-k}^{n}(m-1)^{j}{n \choose j}} (m-1)^{n-k}{n \choose k}\] 
\end{theorem}

We remark that our proof generalizes straightforwardly to prove this also.

A slight modification of our method yields a result which gives us more `geometrical' insight into the problem:

\begin{theorem}
\label{thm:hammingballbound}
Let \(B\) be a Hamming ball of radius \(k\) in \(\{0,1\}^{n}\). If \(\mathcal{A}\) is an irredundant family of \(k\)-subcubes of \(\{0,1\}^{n}\), each with a private vertex in \(B\), then \(|\mathcal{A}| \leq {n \choose k}\).
\end{theorem}
\begin{proof}
By symmetry, we may assume that \(B = [n]^{(\leq k)}\). Let \(\mathcal{A}\) be an irredundant family of \(k\)-subcubes, each with a private vertex in \([n]^{(\leq k)}\). For each subcube \(C \in \mathcal{A}\), choose a private vertex \(w_{C}\in [n]^{(\leq k)}\). Write \(C = \{y \in Q_{n}:\ v_{C} \subset y \subset u_{C}\}\); we will call \(v_{C}\) the `start vertex' of \(C\) and \(u_{C}\) its `end vertex'. Let \(C' = \{y \in Q_{n}:\ w_{C}\subset y \subset u_{C}\}\) be the \((k-|w_{C}|+|v_{C}|)\)-dimensional sub-subcube of \(C\) between the private vertex and the end vertex of \(C\).\\
\\
\textit{Claim:} For any vertex \(x \in [n]^{(k)}\),
\begin{equation}
\label{eq:boundsum2}
\sum_{C \in \mathcal{A}:\ x \in C'}{|v_{C}| + k-|w_{C}| \choose k-|w_{C}|}^{-1} \leq 1
\end{equation}\\
\textit{Proof of Claim:}\\
As before, this is an immediate consequence of Bollob\'as' Inequality. By symmetry, we may assume that \(x = [k]\). Write \(\{C\in \mathcal{A}:\ x \in C'\} = \{C_{1},\ldots,C_{N}\}\). Let \(v_{i} = v_{C_{i}}\) be the start vertex of \(C_{i}\) and \(w_{i} = w_{C_{i}}\) its private vertex. Clearly, \(v_{i},w_{i} \subset [k]\) for every \(i \in [N]\). Notice that \(v_{i} \subset w_{j}\) if and only if \(i=j\), i.e. \(v_{i} \cap ([k]\setminus w_{j}) = \emptyset\) if and only if \(i=j\). Hence, Bollob\'as' Inequality gives:
\[\sum_{i=1}^{N}{|v_{i}|+k-|w_{i}| \choose k-|w_{i}|}^{-1} \leq 1\]
and the claim is proved.

Summing (\ref{eq:boundsum2}) over all \(x \in [n]^{(k)}\), and interchanging the order of summation, we obtain: 

\begin{eqnarray*}
{n \choose k} & \geq & \sum_{x \in [n]^{(k)}} \sum_{\substack{C \in \mathcal{A}:\\ x \in C'}} {|v_{C}| + k-|w_{C}| \choose k-|w_{C}|}^{-1} \\
& = & \sum_{C \in \mathcal{A}} \sum_{ x \in C' \cap [n]^{(k)}}{|v_{C}| + k-|w_{C}| \choose k-|w_{C}|}^{-1}\end{eqnarray*}

For each subcube \(C \in \mathcal{A}\), the \((k-|w_{C}| +|v_{C}|)\)-dimensional subcube \(C'\) contains \({k-|w_{C}| +|v_{C}| \choose k-|w_{C}|}\) vertices \(x \in [n]^{(k)}\), and for each of them contributes \({|v_{C}| + k-|w_{C}| \choose k-|w_{C}|}^{-1}\) to the above sum, i.e. a total of 1. Hence,
\[|\mathcal{A}| = \sum_{C \in \mathcal{A}} \sum_{x \in C' \cap [n]^{(k)}:\ x \in C'}{|v_{C}| + k-|w_{C}| \choose k-|w_{C}|}^{-1}\leq {n \choose k},\]
proving the theorem.
\end{proof}

We have equality in Theorem \ref{thm:hammingballbound} if \(\mathcal{A}\) is the family of all \(k\)-subcubes through the centre of \(B\). Notice that by fixing some choice of private vertices and averaging over all Hamming balls \(B\) of radius \(k\), Theorem \ref{thm:hammingballbound} immediately implies Theorem \ref{thm:meshulambound}.

When \(n=2k+1\), the irredundant family of all \(k\)-subcubes containing either \(\boldsymbol{0}\) or \(\boldsymbol{1}\) has size \(2{n \choose k}\), so we have equality in Theorem \ref{thm:meshulambound} when \(n = 2k+1\).

We have been unable to find a counterexample to Conjecture \ref{conjecture:aharoni}. Notice that by the same projection argument as in Corollary 6 (see later), if the conjecture holds for \(n,k\) then it holds for \(n+1,k+1\), so it suffices to consider the case \(n=2k-1\). For \(n=5,k=3\), the conjecture can be verified by hand, but there are exactly two extremal families up to isomorphism (permuting the coordinates and translating): \(\mathcal{F}_{\boldsymbol{0}}\) and the following family of ten \(3\)-subcubes of \(Q_{5}\), five through \(\boldsymbol{0}\) and five through \(\boldsymbol{1}\). The (unique) private vertices are indicated above the moving coordinates:\\
\\
\((\stackrel{1}{*},\stackrel{0}{*},\stackrel{1}{*},0,0)\)\\*
\((0,\stackrel{1}{*},\stackrel{0}{*},\stackrel{1}{*},0)\)\\*
\((0,0,\stackrel{1}{*},\stackrel{0}{*},\stackrel{1}{*})\)\\*
\((\stackrel{1}{*},0,0,\stackrel{1}{*},\stackrel{0}{*})\)\\*
\((\stackrel{0}{*},\stackrel{1}{*},0,0,\stackrel{1}{*})\)\\*
\((\stackrel{0}{*},\stackrel{1}{*},\stackrel{0}{*},1,1)\)\\*
\((1,\stackrel{0}{*},\stackrel{1}{*},\stackrel{0}{*},1)\)\\*
\((1,1,\stackrel{0}{*},\stackrel{1}{*},\stackrel{0}{*})\)\\*
\((\stackrel{0}{*},1,1,\stackrel{0}{*},\stackrel{1}{*})\)\\*
\((\stackrel{1}{*},\stackrel{0}{*},1,1,\stackrel{0}{*})\)\\
\\
Clearly, this family is not of the form \(\mathcal{F}_{x}\) for any \(x \in \{0,1\}^{5}\). However, we have been unable to find another such example, and we conjecture that for \(n > 5\) and \(k > n/2\), the only irredudant families of \(k\)-subcubes of \(\{0,1\}^{n}\) with size \({n \choose k}\) are of the form \(\mathcal{F}_{x}\) for \(x \in \{0,1\}^{n}\).

The best upper bound for \(n = 2k-1\) is still Meshulam's bound, which in this case is:

\begin{eqnarray*}
M(2k-1,k) & \leq & \frac{2^{2k-1}}{2^{2k-2}+{2k-1\choose k}}{2k-1\choose k}\\
& = & \frac{2}{1+2^{-(2k-2)}{2k-1 \choose k}}{2k-1 \choose k}\\
& = & \frac{2}{1+2(1+o(1))/\sqrt{(2k-1)\pi}}{2k-1 \choose k}\\
& = & 2(1-\Theta(1/\sqrt{k})){2k-1 \choose k}.
\end{eqnarray*}

To construct a large irredundant family when \(k \geq n/2\), one might try just using subcubes containing \(\boldsymbol{0}\) or \(\boldsymbol{1}\), so that the \(k\)-subcubes containing \(\boldsymbol{0}\) have private vertices in \([n]^{(\leq k)}\), and the \(k\)-subcubes containing \(\boldsymbol{1}\) have private vertices in \([n]^{(\geq n-k)}\). However, a surprising linear algebra argument shows that even when \(n=2k\), such a family has size at most \({n \choose k}\):

\begin{theorem}
\label{thm:topbottom}
If \(\mathcal{A}\) is an irredundant family of \(k\)-subcubes of \(\{0,1\}^{2k}\) which contain \(\boldsymbol{0}\) or \(\boldsymbol{1}\), then \(|\mathcal{A}| \leq {2k \choose k}\).
\end{theorem}
\begin{proof}
Let \(\mathcal{A}\) be an irredundant family of \(k\)-subcubes of \(\{0,1\}^{2k}\) which all contain either \(\boldsymbol{0}\) or \(\boldsymbol{1}\). We may assume that \(\mathcal{A}\) is maximal with respect to this condition. For \(v \in [2k]^{(k)}\), we write
\[\mathbb{U}v := \{y: v \subset y \subset [2k]\}\]
for the \(k\)-subcube between \(v\) and \([2k]\).

We partition the vertices of the middle layer \([2k]^{(k)}\) into three sets:
\begin{eqnarray*}
S & = & \{v \in [2k]^{(k)}: \mathbb{P}v,\mathbb{U}v \in \mathcal{A}\};\\
T & = & \{v \in [2k]^{(k)}: \textrm{ exactly one of }\mathbb{P}v\textrm{ and }\mathbb{U}v \textrm{ is in }\mathcal{A}\};\\
R & = & \{v \in [2k]^{(k)}: \mathbb{P}v \notin \mathcal{A},\mathbb{U}v \notin \mathcal{A}\}.
\end{eqnarray*}
Notice that
\[|\mathcal{A}| = {2k \choose k} + |S| - |R|;\]
we must show that \(|S| \leq |R|\).

Write \(S = \{v_{1},\ldots,v_{N}\}\). For each \(v_{i} \in S\), \(\mathbb{P}v_{i}\) must have a private vertex \(w_{i} \in [2k]^{(\leq k-1)}\). If \(|w_{i}| < k-2\), then we may choose \(b_{i}\in [2k]^{(k-1)}\) such that \(w_{i} \subset b_{i} \subset v_{i}\); \(b_{i}\) must also be a private vertex for \(\mathbb{P}v_{i}\), since any subcube containing both \(\boldsymbol{0}\) and \(b_{i}\) must contain \(w_{i}\) as well. Similarly, we may choose a private vertex \(c_{i} \in [2k]^{(k+1)}\) for \(\mathbb{U}v_{i}\). Each point of \(T\) is a private vertex for the subcube in \(\mathcal{A}\) containing it. Let \(\mathcal{B} = \{b_{1},\ldots,b_{N}\}\), and let \(\mathcal{C} = \{c_{1},\ldots,c_{N}\}\). Then we can choose all the private vertices to lie in \(T \cup \mathcal{B} \cup \mathcal{C}\). For each \(i\), let
\[B_{i} = \{x \in [2k]^{(k)}: b_{i} \subset x\},\quad C_{i} = \{x \in [2k]^{k}: x \subset c_{i}\}\]
be the neighbourhoods of \(b_{i}\) and \(c_{i}\) in \([2k]^{(k)}\). First, we claim that
\[\left(\bigcup_{i=1}^{N}B_{i}\right) \cap \left(\bigcup_{i=1}^{N}C_{i}\right) = S \cup R.\]
To see this, take \(x \in (\cup_{i=1}^{N}B_{i}) \cap (\cup_{i=1}^{N}C_{i});\) then \(b_{i} \subset x \subset c_{j}\) for some \(i\) and \(j\). Suppose \(\mathbb{P}x \in \mathcal{A}\); then \(b_{i} \in \mathbb{P}x\), so \(x=v_{i} \in S\), i.e. \(\mathbb{U}x \in \mathcal{A}\) as well. Similarly, if \(\mathbb{U}x \in \mathcal{A}\), then \(\mathbb{P}x \in \mathcal{A}\) as well. Hence, \((\cup_{i=1}^{N}B_{i}) \cap (\bigcup_{i=1}^{N}C_{i}) \subset S \cup R\).

Clearly, \(S \subset (\cup_{i=1}^{N}B_{i}) \cap (\cup_{i=1}^{N}C_{i})\), as \(b_{i} \subset v_{i} \subset c_{i}\) for every \(i\). If \(x \in R\), then by the maximality of \(\mathcal{A}\), \(\mathbb{P}x\) must contain some \(b_{i}\) (otherwise it could be added to \(\mathcal{A}\) to produce a larger irredundant family), and similarly \(\mathbb{U}x\) must contain some \(c_{j}\). Hence, \(x \in (\cup_{i=1}^{N}B_{i}) \cap (\cup_{i=1}^{N}C_{i})\). It follows that \(R \subset (\cup_{i=1}^{N}B_{i}) \cap (\cup_{i=1}^{N}C_{i})\) as well, proving the claim.

For each \(i\), let \(B_{i}' = B_{i} \cap R = B_{i} \setminus S\), and let \(C_{i}'=C_{i} \cap R = C_{i} \setminus S\); then \(B_{i}',C_{i}' \subset R\) for each \(i\). We claim that
\begin{equation}
\label{eq:paritycondition}
|B_{i}' \cap C_{i}'| = 1 \textrm{ for each }i,\textrm{ and } |B_{i}' \cap C_{j}'| = 0 \textrm{ or }2\ \textrm{ for each } i \neq j.
\end{equation}
To see this, first observe that for each \(i\),
\[B_{i} \cap C_{i} = \{x\in [2k]^{(k)}: \ b_{i} \subset x \subset c_{i}\} = \{v_{i},y_{i}\}\]
for some \(y_{i} \in R\), and therefore
\[B_{i}' \cap C_{i}' = \{y_{i}\}.\]
For each \(i \neq j\), if \(b_{i} \nsubseteq c_{j}\), then
\[B_{i} \cap C_{j}=\emptyset\]
and therefore
\[B_{i}' \cap C_{j}' = \emptyset.\]
If \(b_{i} \subset c_{j}\), then \(B_{i} \cap C_{j} = \{x \in [2k]^{(k)}: \ b_{i} \subset x \subset c_{j}\}\) has size 2, and cannot contain a point of \(S\), since if \(b_{i}\subset v_{l}\subset c_{j}\), then \(i=j=l\). Hence, \(B_{i}' \cap C_{j}'\) also has size 2, proving (\ref{eq:paritycondition}). 

We recall the following easy lemma, the \(p=2\) case of which appears in \cite{babaifrankl}:

\begin{lemma}
Let \(p\) be prime. If \(F_{1},\ldots,F_{N},G_{1},\ldots,G_{N} \subset [m]\) are such that
\begin{eqnarray*}
& |F_{i} \cap G_{j}| \equiv 0 \mod p & \forall i \neq j\\
and & |F_{i} \cap G_{i}| \not \equiv 0 \mod p  &\forall i,
\end{eqnarray*}
then
\[N \leq m.\]
\end{lemma}
\begin{proof}
Let \(\chi_{F}\) be the characteristic function of \(F \subset [m]\). Consider it as an element of the \(m\)-dimensional vector space \(\mathbb{F}_{p}^{m}\) over \(\mathbb{F}_{p}\). Observe that \(\{\chi_{F_{1}},\ldots,\chi_{F_{N}}\}\) is linearly independent over \(\mathbb{F}_{p}\). To see this, suppose
\[\sum_{i=1}^{N}r_{i}\chi_{F_{i}} = 0\]
for some \(r_{1},\ldots,r_{N} \in \mathbb{F}_{p}\). Taking the inner product of the above with \(\chi_{G_{j}}\) gives \(r_{j} = 0\). Hence, \(N \leq m\) as required.
\end{proof}
Applying the \(p=2\) case of this lemma to the sets \(B_{1}',\ldots,B_{N}',C_{1}',\ldots,C_{N}' \subset R\) shows that \(|S| \leq |R|\), proving the theorem.
\end{proof} 

We immediately obtain the same result for all \(n \leq 2k\), by induction on \(n\) for fixed codimension \(c=n-k\), using a projection argument:

\begin{corollary}
\label{corr:topbottom}
Let \(n \leq 2k\). If \(\mathcal{A}\) is an irredundant family of \(k\)-subcubes of \(\{0,1\}^{n}\) which contain \(\boldsymbol{0}\) or \(\boldsymbol{1}\), then \(|\mathcal{A}| \leq {n \choose k}\).
\end{corollary}
\begin{proof}
Suppose the result is true for some \(n\) and \(k\) such that \(n \geq 2k\); we will prove it for \(n+1,k+1\). Let \(\mathcal{A}\) be an irredundant family of \((k+1)\)-subcubes of \(\{0,1\}^{n+1}\) which contain \(\boldsymbol{0}\) or \(\boldsymbol{1}\). Let \(\mathcal{A}_{i} = \{C \in \mathcal{A}: C_{i} = *\}\) be the collection of subcubes in \(\mathcal{A}\) with coordinate \(i\) moving; since each subcube has \(k+1\) moving coordinates,
\[\sum_{i=0}^{n+1}|\mathcal{A}_{i}| = (k+1)|\mathcal{A}|.\]
We will show that \(|\mathcal{A}_{i}| \leq {n \choose k}\) for each \(i \in [n+1]\), giving \(|\mathcal{A}| \leq \frac{n+1}{k+1} {n \choose k} = {n+1 \choose k+1}\). Without loss of generality, \(i=n+1\). We project the family \(\mathcal{A}_{n+1}\) of \((k+1)\)-subcubes onto \(\{0,1\}^{n}\): let \(\mathcal{A}_{n+1}' = \{C' : C \in \mathcal{A}_{n+1}\}\), where \(C'\) is the \(k\)-subcube of \(\{0,1\}^{n}\) produced by projecting \(C\) onto \(\{0,1\}^{n}\), i.e. deleting the \((n+1)\)-coordinate of \(C\) (which is a \(*\)). Clearly, \(\mathcal{A}_{n+1}'\) is a collection of \(|\mathcal{A}_{n+1}|\) \(k\)-subcubes of \(\{0,1\}^{n}\) through \(\boldsymbol{0}\) or \(\boldsymbol{1}\). It is also irredundant, as the projection of a private vertex of \(C\) in \(\mathcal{A}_{n+1}\) is clearly a private vertex for \(C'\) in \(\mathcal{A}_{n+1}'\). Hence, by the induction hypothesis, \(|\mathcal{A}_{n+1}'| \leq {n \choose k}\), giving the result.
\end{proof}

Notice that we do not have uniqueness of the extremal families in Theorem \ref{thm:topbottom} for any value of \(k\): as well as taking \(\mathcal{A} = \mathcal{F}_{\boldsymbol{0}}\) or \(\mathcal{F}_{\boldsymbol{1}}\), any family \(\mathcal{A}\) containing exactly one of \(\mathbb{P}x,\mathbb{U}x\) for each \(x \in [2k]^{(k)}\) is extremal. Slightly more surprisingly, we do not have uniqueness (in Corollary \ref{corr:topbottom}) for \(n=5,k=3\) either: consider the irredundant family of ten \(3\)-subcubes of \(\{0,1\}^{5}\), five through \(\boldsymbol{0}\) and five through \(\boldsymbol{1}\), exhibited earlier.\\

\section{Lower bounds}

\subsubsection*{The case \(n=2k\).}
Now, returning to general irredundant families, what can we say about the case \(n = 2k\)? Meshulam's bound gives:
\begin{eqnarray*}
M(2k,k) & \leq & \frac{2}{1+2^{-2k}{2k \choose k}}{2k \choose k}\\
& = & \frac{2}{1+(1+o(1))/\sqrt{2\pi k}}{2k \choose k}\\
& = & 2(1-\Theta(1/\sqrt{k})){2k \choose k}
\end{eqnarray*}

Our lower bound (\ref{eq:randomlowerbound}) no longer beats \(\mathcal{F}_{\boldsymbol{0}}\), since it only gives
\[M(2k,k) \geq \beta(1-\beta)^{(1-\beta)/\beta} 2^{2k} = (1+o(1))\frac{\beta}{e(1-\beta)}2^{2k} = (1+o(1))\frac{2}{e}{2k \choose k}.\]
Notice that \(\mathcal{F}_{\boldsymbol{0}}\) is a maximal irredundant family. We know from Theorem \ref{thm:topbottom} that any irredundant family of \(k\)-subcubes in which each goes through either \(\boldsymbol{0}\) or \(\boldsymbol{1}\) has size at most \({2k \choose k}\); we now exhibit a maximal such family \(\mathcal{B}\) which is not maximal irredundant.

Let \(\mathcal{B}_{0} = \{\mathbb{P}x: 1 \in x\}\) be the collection of \(k\)-subcubes containing the line \((*,0,0,...,0)\), and \(\mathcal{B}_{1} = \{\mathbb{U}x: n \notin x\}\) the collection containing \((1,1,...,1,*)\). Consider the family \(\mathcal{B} = \mathcal{B}_{0} \cup \mathcal{B}_{1}\); it has size \(|\mathcal{B}| = 2{2k-1 \choose k-1} = {2k \choose k}\); we will show that it is irredundant and not maximal. What are the \(\mathcal{B}\)-private vertices of each subcube \(C \in \mathcal{B}\)? Write \(C_{i}\) for the symbol (\(0,1\) or \(*\)) in the \(i\)-coordinate of the subcube \(C\). There are 4 different types of subcubes in \(\mathcal{B}\) to consider:
\begin{itemize}
\item \(C \in \mathcal{B}_{0}\) with \(C_{n} = 0\), e.g. \(C=\)\\*
\((*,*,\ldots,*,*,0,\ldots,0)\) has \(\mathcal{B}_{0}\)-private vertices\\*
\((*,1,\ldots,1,1,0,\ldots,0)\);\\*
\((1,1,\ldots,1,1,0,\ldots,0) \in (1,1,\ldots,1,1,*,\ldots,*) \in \mathcal{B}_{1}\), but\\*
\((0,1,\ldots,1,1,0,\ldots,0) \in [n]^{(k-1)}\) so is not in any \(D \in \mathcal{B}_{1}\), so is the unique \(\mathcal{B}\)-private vertex of \(C\).
\item \(C \in \mathcal{B}_{0}\) with \(C_{n} = *\): e.g. \(C =\)\\*
\((*,*,\ldots,*,0,\ldots,0,*)\) has \(\mathcal{B}_{0}\)-private vertices\\*
\((*,1,\ldots,1,0,\ldots,0,1)\);\\*
this line has \(k\) fixed 0's in coordinates \(\{2,\ldots,n-1\}\) whereas each \(D \in \mathcal{B}_{1}\) has at most \(k-1\) \(*\)'s in this range, hence this line is disjoint from \(\mathcal{B}_{1}\) and both its vertices are the unique \(\mathcal{B}\)-private vertices of \(C\).
\item \(C \in \mathcal{B}_{1}\) with \(C_{1} = 1\): e.g. \(C=\)\\*
\((1,*,\ldots,*,1,\ldots,1,*)\) has \(\mathcal{B}\)-private vertex\\*
\((1,0,\ldots,0,1,\ldots,1,1)\)
\item \(C \in \mathcal{B}_{1}\) with \(C_{1} = *\): e.g. \(C=\)\\*
\((*,*,\ldots,*,1,\ldots,1,*)\) has \(\mathcal{B}\)-private vertices\\*
\((0,0,\ldots,0,1,\ldots,1,*)\)
\end{itemize}
Notice that
\[\cup_{D \in \mathcal{B}_{0}}D = [n]^{(\leq k-1)}\cup\{x \in [n]^{(k)}: 1 \in x\}\]
and
\[\cup_{D \in \mathcal{B}_{1}}D = [n]^{(\geq k+1)}\cup\{x \in [n]^{(k)}: n \notin x\}\]
Hence,
\[\{0,1\}^{n} \setminus \cup_{D \in \mathcal{B}}D = \{x \in [n]^{(k)}: 1 \notin x,n \in x\}\]
Now let \(E\) be any \(k\)-subcube with \(E_{1} = 0,E_{n} = 1\).\\
\textit{Claim}: \(\mathcal{B} \cup \{E\}\) is also irredundant.\\
\textit{Proof of Claim}: If \(E\) has \(s\) 0's and \(t\) 1's in coordinates \(\{2,\ldots,n-1\}\), where \(s+t=k-2\), then setting \(k-1-t\) \(*\)'s = 1 and the other \(t+1\) \(*\)'s = 0, we find an \(x \in E\cap[n]^{(k)}: 1 \notin x,n \in x\), i.e. a \(\mathcal{B}\)-private vertex for \(E\). We must now check that each of the above types of subcube in \(\mathcal{B}\) has a \(\mathcal{B}\)-private vertex not in \(E\):
\begin{itemize}
\item \(C \in \mathcal{B}_{0}\) with \(C_{n} = 0\): disjoint from \(E\), so the \(\mathcal{B}\)-private vertex will do.
\item \(C \in \mathcal{B}_{0}\) with \(C_{n} = *\): choose the \(\mathcal{B}\)-private vertex with 1-coordinate 1.
\item \(C \in \mathcal{B}_{1}\) with \(C_{1} = 1\): disjoint from \(E\), so the \(\mathcal{B}\)-private vertex will do.
\item \(C \in \mathcal{B}_{1}\) with \(C_{1} = *\): choose the \(\mathcal{B}\)-private vertex with \(n\)-coordinate 0.
\end{itemize}
This proves the claim. 
How many such subcubes can we add on? We can certainly add on the family:
\[\mathcal{E} = \{E: E_{1} = 0,E_{n} = 1, E_{2} = *,E_{i} = 0\; \textrm{or}\; * \forall i \neq 1,2\; \textrm{or}\; n\}\]
e.g. the subcube\\
\((0,*,0,\ldots,0,*,\ldots,*,1)\) has private vertex\\
\((0,1,0,\ldots,0,1,\ldots,1,1)\).\\
Hence,
\[M(2k,k) \geq {2k \choose k}+ {2k-3 \choose k-1} = (1+\tfrac{1}{8}+o(1)){2k \choose k}\]
but we still have a gap of \(\tfrac{7}{8}\) between the constants in our lower and upper bounds.

Notice the sharp drop by a factor of order \(\sqrt{n}\) from \(M(n,\lfloor \gamma n \rfloor) = \Theta_{\gamma}(2^{n})\) for \(\gamma \in (0,\tfrac{1}{2})\) to
\[M(n,\lfloor n/2 \rfloor) \leq 2 {n \choose \lfloor n/2 \rfloor} = 2(1+o(1))\frac{2^{n}}{\sqrt{\pi n}}\]

\subsubsection*{The case \(k < \tfrac{1}{2}n\)}
When \(k < \tfrac{1}{2}n\), we can construct an irredundant family by taking a union of \(\mathcal{F}_{v}\)'s: choose a maximum \((2k+1)\)-separated subset \(S \subset \{0,1\}^{n}\) (i.e. a maximum \(k\)-error correcting code) and let
\[\mathcal{F}_{S} = \cup_{v \in S}\mathcal{F}_{v}\]
be the family of all \(k\)-subcubes containing a point of \(S\); then
\[|\mathcal{F}_{S}| = |S|{n \choose k}.\]
When there is a subset \(S \subset \{0,1\}^{n}\) such that the Hamming balls of radius \(k\) centred on the vertices of \(S\) {\em partition} \(\{0,1\}^{n}\) (i.e. a perfect \(k\)-error correcting code), 
\[|\mathcal{F}_{S}| = \frac{2^n}{\sum_{i=0}^{k}{n \choose i}}{n \choose k}\]
which exactly matches Meshulam's bound.

It is known that there is a perfect \(k\)-error correcting code in \(\{0,1\}^{n}\) precisely in the following cases (see \cite{perfectcodes}):
\begin{itemize}
\item \(k=1\), \(n+1\) is a power of 2 (take any Hamming code)
\item \(k = 3\), \(n=23\) (take the Golay code)
\item \(n = 2k+1\) (take a `trivial' code, two vertices of distance \(n\) apart)
\end{itemize}
so in these cases, we have equality in Meshulam's bound:
\[M(n,k) = \frac{2^n}{\sum_{l=0}^{k}{n \choose l}}{n \choose k}.\]

First, consider the case \(k=1\); a 1-subcube is simply an edge of \(\{0,1\}^{n}\). Meshulam's bound is
\[M(n,1) \leq \tfrac{n}{n+1}2^{n}.\]
Kabatyanskii and Panchenko \cite{1packings} proved the existence of asymptotically perfect packings of 1-balls into \(\{0,1\}^{n}\), namely that there is a packing of \[\frac{2^{n}}{n+1}(1-O(\ln \ln n / \ln n))\]
1-balls into \(\{0,1\}^{n}\). Taking all edges through the centre of each ball gives an irredundant family of size
\[\frac{n}{n+1}2^{n}(1-O(\ln \ln n / \ln n)) = 2^{n}(1-O(\ln \ln n / \ln n))\]
We can in fact improve on this with the following `product' construction. Let \(s \in \mathbb{N}\) be maximal such that \(2^{s}-1 \leq n\); write \(n = m+r\) where \(m = 2^{s}-1\). Take a perfect packing of 1-balls into \(\{0,1\}^{m}\) and take all edges through the centre of each ball, producing an irredundant family \(\mathcal{B}\) in \(\{0,1\}^{m}\) of size \(\tfrac{m}{m+1}2^{m}\). Writing \(\{0,1\}^{n} = \{0,1\}^m \times \{0,1\}^{r}\), let \(\mathcal{A}\) be the family consisting of a copy of \(\mathcal{B}\) in each of the \(2^{r}\) disjoint copies of \(\{0,1\}^{r}\); \(|\mathcal{A}| = \tfrac{m}{m+1}2^{n}\). Notice that \(m = 2^{s}-1 \geq \tfrac{1}{2}n\), since otherwise \(2^{s+1}-1 \leq n\), contradicting the maximality of \(s\). Hence, \(|\mathcal{A}| \geq \tfrac{n}{n+2}2^{n}\), and we have
\[M(n,1) \geq \frac{n}{n+2}2^{n}\quad \forall n \in \mathbb{N},\]
so
\[M(n,1) = 2^{n}(1-\Theta(1/n)).\]

What about for \(k\) fixed and \(n\) growing? It is a longstanding open problem in coding theory to determine whether, for \(k\) fixed, there is an asymptotically perfect packing of \(k\)-balls into \(\{0,1\}^{n}\), i.e. a packing of
\[\frac{2^{n}}{\sum_{i=0}^{k} {n \choose i}}(1-o(1))\]
\(k\)-balls into \(\{0,1\}^{n}\); given such, by taking all \(k\)-subcubes through the centre of each ball, we would immediately obtain an irredundant family of size
\[\frac{{n \choose k}}{\sum_{l=0}^{k} {n \choose l}}2^{n}(1-o(1)) = 2^{n}(1-o(1))\]
However, this conjecture remains unsolved for all \(k > 1\).

Moreover, for \(k = \Omega(n)\), the approach outlined above can only give a relatively small irredundant family. Corr\'adi and Katai \cite{corradi} proved the following:

\begin{theorem}[Corr\'adi-Katai, 1969]
Let \(S \subset \{0,1\}^{n}\) be an \((n/2)\)-separated set; then
\begin{itemize}
\item \(|S| \leq n+1\) if \(n\) is odd
\item \(|S| \leq n+2\) if \(n \equiv 2 \mod 4\)
\item \(|S| \leq 2n\) if \(n \equiv 0 \mod 4\)
\end{itemize}
\end{theorem}

(For a proof of this, we refer the reader for example to \cite{bollobas} \S 10.)

So we see that, for example, any \((2k+1)\)-separated family \(S\) of vertices in \(Q_{4k}\) must have \(|S| \leq 8k\), and so taking all \(k\)-subcubes through each of these vertices only gives 
\[|\mathcal{F}_{S}| \leq 8k{4k \choose k} \leq 8k\exp\left(-\frac{4k}{32}\right)2^{4k}.\]

We now improve on this using a probabilistic method. The idea is to take a random subset \(S \subset \{0,1\}^{n}\) where each vertex is present independently with some fixed probability \(p\); for each vertex \(w \in \{0,1\}^{n}\) of (Hamming) distance \(k\) from \(S\), we choose a \(k\)-subcube \(C_{w}\) between \(w\) and some vertex of \(S\), giving a random irredundant family of \(k\)-subcubes \(\mathcal{A} = \{C_{w}:d(w,S)=k\}\); the expected size of this family is then a lower bound for \(M(n,k)\).

\begin{theorem}
\label{thm:randomlowerbound}
For any \(k \leq n\), there exists an irredundant family of \(k\)-subcubes of \(\{0,1\}^{n}\) of size at least
\[\beta(1-\beta)^{(1-\beta)/\beta}2^{n},\]
where 
\[\beta = \beta_{n,k} := \frac{{n \choose k}}{\sum_{i=0}^{k}{n \choose i}}.\]
\end{theorem}
\begin{proof}
Let \(S\) be a random set of vertices in \(\{0,1\}^{n}\) where each vertex is present independently with probability \(p\) (to be chosen later). Consider the random set of vertices
\[W=\{x \in \{0,1\}^{n}: d(x,S) = k\},\]
where \(d(x,y) = |x \Delta y|\) denotes the Hamming distance between \(x\) and \(y\). For each \(w \in W\), choose any \(x_{w} \in S\) such that \(|w \Delta x_{w}| = k\), and let \(C_{w}\) be the \(k\)-subcube between \(x_{w}\) and \(w\), i.e. 
\[C_{w} = \{y \in \{0,1\}^{n}:\ y \Delta w \subset x_{w} \Delta w\}.\]
Consider the random family of \(k\)-subcubes
\[\mathcal{A} = \{C_{w}:w \in W\}.\]
Note that the subcubes \(C_{w}\) are pairwise distinct: \(x_{w}\) is the unique point of \(S\) in \(C_{w}\), and \(w\) is the `opposite' point, so \(C_{w}\) determines \(w\). Moreover, \(\mathcal{A}\) is irredundant, since \(w\) is a private vertex of \(C_{w}\). (If \(w \in C_{w'}\), then \(|x_{w'} \Delta w| \leq k\), so \(|x_{w'} \Delta w| = k\), so \(w\) is the unique vertex in \(C_{w'}\) of distance \(k\) from \(x_{w'}\), so \(w=w'\).) We now calculate the expectation of the random variable \(|\mathcal{A}|=|W|\). A vertex \(v \in \{0,1\}^{n}\) is in \(W\) if and only if the \((k-1)\)-ball around \(v\) contains no vertices of \(S\) but the \(k\)-ball around \(v\) does contain a vertex of \(S\); the probability of this event is
\[(1-p)^{\sum_{i=0}^{k-1}{n \choose i}}-(1-p)^{\sum_{i=0}^{k}{n \choose i}}.\]
Hence, the expected size of \(\mathcal{A}\) is
\[\mathbb{E}|\mathcal{A}| = 2^{n}\left((1-p)^{\sum_{i=0}^{k-1}{n \choose i}}-(1-p)^{\sum_{i=0}^{k}{n \choose i}}\right).\]
Let
\[\beta =\beta_{n,k} := \frac{{n \choose k}}{\sum_{i=0}^{k}{n \choose i}},\quad t := (1-p)^{\sum_{i=0}^{k}{n \choose i}};\]
then
\[\mathbb{E}|\mathcal{A}| = 2^{n}(t^{1-\beta}-t).\]
The function
\begin{eqnarray*}
f: [0,1] & \to & \mathbb{R};\\
t & \mapsto & t^{1-\beta}-t
\end{eqnarray*}
attains its maximum of
\[\beta(1-\beta)^{(1-\beta)/\beta}\]
at
\[t = (1-\beta)^{1/\beta}.\]
Hence, choosing \(p\) such that
\[(1-p)^{\sum_{i=0}^{k}{n \choose i}} = (1-\beta)^{1/\beta},\]
our random irredudant family has expected size
\[\mathbb{E}|\mathcal{A}| = \beta(1-\beta)^{(1-\beta)/\beta}2^{n}.\]
Hence, there exists an irredundant family of size at least this, proving the theorem.
\end{proof}

Combining this with Meshulam's bound, we see that
\begin{equation}
\label{eq:combination}
\beta(1-\beta)^{(1-\beta)/\beta} 2^{n} \leq M(n,k) \leq \beta 2^{n}.
\end{equation}
The ratio between the lower and upper bound above is
\[g(\beta) := (1-\beta)^{(1-\beta)/\beta}.\]
Observe that \(g'(\beta)>0 \ \forall \beta \in (0,1)\), so \(g\) is strictly increasing on \((0,1)\). Note that
\[\ln(g(\beta)) = \frac{1-\beta}{\beta}\ln(1-\beta) \to -1\quad \textrm{as}\ \beta \to 0,\]
so \(g(\beta) \to 1/e\) as \(\beta \to 0\); \(\ln(g(\beta)) \to 0\) as \(\beta \to 1\), so \(g(\beta) \to 1\) as \(\beta \to 1\). Hence, \(1/e \leq g(\beta) \leq 1\ \forall \beta \in (0,1)\), so the ratio between the upper and lower bounds above never exceeds \(e\). We believe that the upper bound is closer to the true value, but we have been unable to improve our lower bound.

If \(k=o(n)\), then \(\beta = 1-o(1)\). Let
\[\eta = 1-\beta = \frac{\sum_{i=0}^{k-1}{n \choose i}}{\sum_{i=0}^{k}{n \choose i}};\]
then \(\eta = o(1)\).

Theorem \ref{thm:randomlowerbound} implies that
\[M(n,k) \geq (1-\eta)\eta^{\eta/(1-\eta)}2^{n} = (1-O(\eta \ln(1/\eta))) 2^{n};\]
which asymptotically matches the upper bound from Meshulam's theorem,
\[M(n,k) \leq \beta 2^{n} = (1-\eta)2^{n}.\]

If \(k=\lfloor \gamma n\rfloor\) for some \(\gamma \in (0,\tfrac{1}{2})\), using the fact that as \(l\) decreases from \(k-1\) to 0, \({n \choose l}\) decreases geometrically, we obtain
\[\beta_{n,\lfloor \gamma n \rfloor} = (1+o(1))\frac{1-2\gamma}{1-\gamma};\]
substituting this into (\ref{eq:combination}) gives: \[(1+o(1))\left(\frac{\gamma}{1-\gamma}\right)^{\frac{\gamma}{1-2\gamma}}\left(\frac{1-2\gamma}{1-\gamma}\right)2^{n} \leq M(n,\lfloor \gamma n \rfloor) \leq (1+o(1))\left(\frac{1-2\gamma}{1-\gamma}\right)2^{n}.\]
Hence, we see that
\[M(n, \lfloor \gamma n \rfloor) = \Theta_{\gamma}(2^{n}).\]

Comparing this with 
\[M(n,\lfloor n/2 \rfloor) = \Theta\left({n \choose \lfloor n / 2 \rfloor}\right) = \Theta(2^{n}/\sqrt{n}),\]
we see that \(M(n,\lfloor \gamma n \rfloor)\) experiences a drop in its order of magnitude at \(\gamma = 1/2\).

\section{Conclusion}
To conclude, we believe Conjecture \ref{conjecture:aharoni} to be true, but that new ideas would be required to prove it for all \(k > n/2\). The problem seems at first glance to be ideal for tackling using the methods of linear algebra, but we have only been able to obtain a sharp result using such methods under the additional constraint of all the subcubes going through \(\boldsymbol{0}\) or \(\boldsymbol{1}\). All the above-mentioned proofs of Meshulam's bound involve considering separately certain subfamilies of an irredundant family, and then averaging; to prove the conjecture when \(k\) is close to \(n/2\), one would need to take into account how an efficient arrangement in one region of \(\{0,1\}^{n}\) is incompatible with efficient arrangements in other parts. The fact that Meshulam's bound is tight for \(n=2k+1\) indicates that the ideas used to prove it will probably not help to approach the conjecture when \(k\) is close to \(n/2\).

If Conjecture \ref{conjecture:aharoni} turns out to be true, it would also be of interest to determine when the only extremal families are the \(\mathcal{F}_{x}\)'s; we conjecture this to be the case for all \(n > 5\). It may also be possible to close the gap between the lower and upper bounds in (\ref{eq:combination}) for \(k < n/2\), though we consider it fortunate that there is only a constant gap between our `random' lower bound and Meshulam's `combinatorial' upper bound.

\end{document}